\documentclass[psamsfonts,reqno,12pt,a4]{amsart}

\usepackage{amssymb,amsfonts}
\usepackage[all,arc]{xy}
\usepackage{enumerate}
\usepackage{mathrsfs}

\usepackage[dvipdfm]{graphicx}
\usepackage[all,arc]{xy}

\usepackage{amssymb, amscd}

\usepackage{url}
\usepackage[margin=3.4cm, paperheight=29.5cm]{geometry}

\newtheorem{thm}{Theorem}[section]
\newtheorem{cor}[thm]{Corollary}
\newtheorem{prop}[thm]{Proposition}
\newtheorem{lem}[thm]{Lemma}

\theoremstyle{definition}
\newtheorem{defn}[thm]{Definition}

\newtheorem{exmp}[thm]{Example}

\newtheorem{prob}[thm]{Problem}

\theoremstyle{remark}
\newtheorem{rem}[thm]{Remark}

\makeatletter
\let\c@equation\c@thm
\makeatother
\numberwithin{equation}{section}

\title[Equipartitions of  measures by hyperplanes]{Computational topology of \\ equipartitions   by hyperplanes }
\author{Rade T.\ \v Zivaljevi\' c}

\address{Mathematical Institute SASA, Knez Mihailova 36, 11001 Belgrade, Serbia}
\email{rade@turing.mi.sanu.ac.rs}
\thanks{Supported by the Serbian Ministry of Science and Technology, Grants 174020  and  174034.}

\date{February 5, 2014 (first version, April 10, 2011)}

\begin{document}

\begin{abstract}
We compute a primary cohomological obstruction to the existence of
an equipartition for $j$ mass distributions in $\mathbb{R}^d$ by
two hyperplanes in the case $2d-3j = 1$. The central new result is
that such an equipartition always exists if $d=6\cdot 2^k +2$ and
$j=4\cdot 2^k+1$ which for $k=0$ reduces to the main result of the
paper P.~Mani-Levitska et al., Topology and combinatorics of
partitions of masses by hyperplanes, {\em Adv.\ Math.} 207 (2006),
266--296. The theorem follows from a Borsuk-Ulam type result
claiming the non-existence of a $\mathbb{D}_8$-equivariant map $f
: S^{d}\times S^d\rightarrow S(W^{\oplus j})$ for an associated
real $\mathbb{D}_8$-module $W$. This is an example of a genuine
combinatorial geometric result which involves
$\mathbb{Z}/4$-torsion in an essential way and cannot be obtained
by the application of either Stiefel-Whitney classes or
cohomological index theories with $\mathbb{Z}/2$ or $\mathbb{Z}$
coefficients. The method opens a possibility of developing an
``effective primary obstruction theory'' based on $G$-manifold
complexes, with applications in geometric combinatorics, discrete
and computational geometry, and computational algebraic topology.
\end{abstract}

\maketitle

\section{Introduction}

\subsection{Computational topology}\label{sec:computational}

Algebraic topology as a tool ``useful for solving discrete
geometric problems of relevance to computing and the analysis of
algorithms'' was in \cite{Z04} isolated as one of important themes
of the emerging `applied and computational topology'. Since the
appearance of the review \cite{Z04}, the field as a whole has
undergone a rapid development, the progress has been made in many
open problems, and the {\em configuration space/test map} paradigm
\cite[Section~14.1]{Z04} has maintained its status as one of
central schemes for applying equivariant topological methods to
combinatorics and discrete geometry.

\medskip
Our objective is twofold: (a) to report a progress
(Theorems~\ref{thm:main-1} and \ref{thm:main-2}) on the
`equipartitions of masses by hyperplanes' problem, which has been
after \cite{Ram} one of benchmark problems for applying
topological methods in computational geometry, (b) to explore a
new scheme for developing `effective primary obstruction theory'
based on $G$-manifold complexes (see
Section~\ref{sec:app-comput-top} in Appendix II for an outline).

\subsection{The equipartition problem}\label{subsec:intro-1}

An old problem in combinatorial geometry is to determine when $j$
measurable sets in $\mathbb{R}^d$ admit an {\em equipartition} by
a collection of $h$ hyperplanes; if this is possible we say that
the triple $(d,j,h)$ is admissible. In this generality the problem
was formulated by Gr\"{u}nbaum in \cite{Gru} but the problem
clearly stems from the classical results of Banach and Steinhaus
\cite{Ste} and Stone and Tukey \cite{Sto-Tuk} on the ``ham
sandwich theorem''.

Among the highlights of the theory of hyperplane partitions of
measurable sets (measures) are Hadwiger's equipartition of a
single mass distribution in $\mathbb{R}^3$ by three hyperplanes
\cite{Hadw}, results of Ramos \cite{Ram} who proved that
$(5,1,4),(9,3,3), (9,5,2)$ are admissible triples, and a result of
Mani-Levitska et al.\ \cite{MVZ} who established a
$2$-equipartition for $5$ measures in $\mathbb{R}^8$. More recent
is the proof of the existence of a $4$-equipartition in
$\mathbb{R}^4$ for measures that admit some additional symmetries
(\v Zivaljevi\' c \cite{Z08}) and a result of Matschke \cite{Mat}
describing a general reduction procedure for admissible triples.

For an account of known results, history of the problem, and an
exposition of equivariant topological methods used for its
solution, the reader is referred to \cite{MVZ}. The landmark paper
of E.~Ramos \cite{Ram} is a valuable source of information and a
link with earlier results on the discrete and computational
geometry of half-space range queries (D.~Dobkin, H.~Edelsbrunner,
M.~Paterson, A.~Yao, F.~Yao). The chapter ``Topological Methods''
in CRC Handbook of Discrete and Computational Geometry \cite{Z04},
provides an overview of the general {\em configuration space/test
map}-scheme for applying equivariant topological methods on
problems of geometric combinatorics and discrete and computational
geometry.

\begin{defn}
\label{defn:equi} Suppose that ${\mathcal M} = \{\mu_1, \mu_2,
\ldots , \mu_j\}$ is a collection of $j$ {\em continuous mass
distributions} (measures) defined in $\mathbb{R}^d$, meaning that
each $\mu_j$ is a finite, positive, $\sigma$-additive Borel
measure, absolutely continuous with respect to Lebesgue measure.
If ${\mathcal H} = \{H_i\}_{i=1}^k$ is a collection of $k$
hyperplanes in $\mathbb{R}^d$ in general position, the connected
components of the complement $\mathbb{R}^d\setminus\cup~{\mathcal
H}$ are called (open) $k$-orthants. The definition can be clearly
extended to the case of degenerate collections ${\mathcal H}$,
when some of the $k$-orthants are allowed to be empty. A
collection ${\mathcal H}$ is an {\em equipartition}, or more
precisely a $k$-equipartition for ${\mathcal M}$ if $$\mu_i(O) =
\mu_i(\overline{O\mathstrut}) =
{\frac{1}{2^k}}\mu_i(\mathbb{R}^d)$$ for each of the measures
$\mu_i\in {\mathcal M}$ and for each $k$-orthant $O$ associated to
${\mathcal H}$.
\end{defn}
A triple $(d,j,k)$ is called {\em admissible} if for each
collection $\mathcal{M}$ of $j$ continuous measures on
$\mathbb{R}^d$ there exists an equipartition of $\mathcal{M}$ by
$k$ hyperplanes. It is not difficult to see that if $(d,j,k)$ is
admissible that $(d+1,j,k)$ is also admissible\footnote{More
general reduction results for admissible triples can be found in
\cite{Mat} and \cite{Ram}.} which motivates the following general
problem.

\begin{prob}\label{prob:general}
The general measure equipartition problem is to determine or
estimate the function
$$\Delta(j,k) := \mbox{\rm min}\{d \mid (d,j,k)\, \mbox{ is admissible} \}$$
or equivalently to find the minimum dimension $d$ such that for
each collection $\mathcal{M}$ of $j$ continuous measures on
$\mathbb{R}^d$, there exists an equipartition of $\mathcal{M}$ by
$k$ hyperplanes.
\end{prob}

\section{New results}\label{sec:new}

Theorem~\ref{thm:main-1} is our central new result about
equipartitions of masses by two hyperplanes. For $k=0$ it reduces
to the result that the triple $(8,5,2)$ is admissible, which is
the main result of \cite{MVZ}. The reader is referred to
references \cite{Ram,MVZ,Mat} for an analysis explaining the
special role of equipartition problems associated to the triples
of the form $(d,j,2)$ where $2d-3j=1$. We emphasize that the
admissibility of all triples listed in Theorem~\ref{thm:main-1}
cannot be established either by the parity count results of Ramos
\cite{Ram}, or by the use of both Stiefel-Whitney characteristic
classes and the ideal-valued cohomology index theory with
$\mathbb{Z}/2$-coefficients.

\begin{thm}\label{thm:main-1} Each collection of $j=4\cdot 2^k+1$ measures in
$\mathbb{R}^d$ where $d=6\cdot 2^k +2$ admits an equipartition by
two hyperplanes. In light of the lower bound $\Delta(j,2)\geq
3j/2$, established by Ramos in \cite{Ram}, this result implies
that for each integer $k\geq 0$,
\begin{equation}\label{eqn:main-1}
\Delta(4\cdot 2^k+1,2)=6\cdot 2^k +2.
\end{equation}
\end{thm}

\medskip
Theorem~\ref{thm:main-1} is deduced from the following Borsuk-Ulam
type result about maps equivariant with respect to the dihedral
group actions.

\begin{thm} \label{thm:main-2}
There does not exist a $\mathbb{D}_8$-equivariant map
\begin{equation}\label{eqn:main-2}
f : S^d\times S^d \rightarrow S(W^{\oplus j})
\end{equation}
where $\mathbb{D}_8$ is the dihedral group of order eight,
$d=6\cdot 2^k +2, j=4\cdot 2^k+1$ for some integer $k\geq 0$, and
$W$ is the representation space of the real $3$-dimensional
$\mathbb{D}_8$-representation described in
Section~\ref{sec:notation}.

Moreover, a first obstruction to the existence of
(\ref{eqn:main-2}) lies in the (special) equivariant cohomology
group $\mathcal{H}^{2d-1}_{\mathbb{D}_8}(S^d\times S^d,
\mathcal{Z})\cong \mathbb{Z}/4$, described in
Section~\ref{sec:appendix} (see
Definition~\ref{defn:adm-filtration} and
Remark~\ref{rem:oznaka-cohom-gp}) and evaluated in
Sections~\ref{sec:fragment} and \ref{sec:evaluation}, where
$\mathcal{Z} = H_{2d-2}(S(W^{\oplus j}); \mathbb{Z})$ and
$2d-3j=1$. The obstruction vanishes unless
$$
d=6\cdot 2^k +2 \quad \mbox{\rm and}\quad j=4\cdot 2^k+1
$$
when it turns out to be equal to $2X$ where $X$ is a generator of
the group $\mathbb{Z}/4$.
\end{thm}

\medskip\noindent
The reader is referred to Section~\ref{sec:summary} for an outline
and overall strategy of proofs of Theorems~\ref{thm:main-1} and
\ref{thm:main-2}. A broader perspective on the method applied in
the proof of Theorem~\ref{thm:main-2}, emphasizing its relevance
for computational obstruction theory, is offered in
Section~\ref{sec:appendix} (Subsection~\ref{sec:app-comput-top}).

\medskip\noindent
{\bf Caveat:}  We emphasize that the special equivariant
cohomology groups $\mathcal{H}^{\ast}_{G}(X; M)$, used in
Theorem~\ref{thm:main-2}, are in general different from the usual
equivariant cohomology groups of a $G$-$CW$ complex $X$ (as
described in \cite{tDieck87}). Nevertheless in many cases they are
easier to compute and still may contain non-zero classes which can
detect a non-trivial obstruction for the existence of an
equivariant map. For their definition and main properties the
reader is referred to Section~\ref{sec:appendix}.

\begin{rem}
The fact that the obstruction $2X\in \mathbb{Z}/4$ is divisible by
$2$ is precisely the reason why Theorem~\ref{thm:main-2} is not
accessible by the methods based on $\mathbb{Z}/2$-coefficients
(parity count \cite{Ram}, Stiefel-Whitney classes,
$\mathbb{D}_8$-equivariant index theory with
$\mathbb{Z}/2$-coefficients). As the elaborate spectral sequences
calculations \cite{Bla-Zie} demonstrate, the methods of
$\mathbb{D}_8$-equivariant index theory with
$\mathbb{Z}$-coefficients are not sufficient either. This may be
somewhat accidental since, as shown in \cite{Go-Lan} and
\cite{Do-Go-Lan}, the $\mathbb{Z}/4$-torsion is often present in
related cohomology calculations.
\end{rem}

\begin{rem}
In light of the reduction procedure of Matschke \cite{Mat}, who
proved the inequality $\Delta(j,k)\leq \Delta(j+1,k)-1$, it is
interesting to test if Theorem~\ref{thm:main-1} generates some new
admissible triples aside from those implied by the inequality
(\ref{eqn:main-1}). The answer is no, since the inequality
$\Delta(2^{k+1},2)\leq 3\cdot 2^k$ was established already by
Ramos in \cite{Ram}.
\end{rem}

\begin{rem} In the cases when (by Theorem~\ref{thm:main-2}) the
first obstruction vanishes, it is still possible that the
secondary obstruction for the existence of the map
(\ref{eqn:main-2}) is non-zero. Both the calculation of the
secondary obstruction and the detection of new admissible triples
$(d,j,h)$ is an interesting open problem.
\end{rem}

\section{Related results and background information}
\label{sec:old}

Results about partitions of measures by hyperplanes have found
numerous applications. The survey \cite{Z04}, which appeared in
2004, is a good source of information about the results obtained
before its publication. However new applications have emerged in
the meantime so we include a brief review of some of the most
interesting recent developments illustrating the relevance and
importance of hyperplane partitions for different areas of
mathematics.

\subsection{Polyhedral partitions of measures}

Equipartitions of measures by $k$-orth\-ants
(Definition~\ref{defn:equi}) are a special case of equipartitions
into polyhedral regions.

\medskip
Very interesting polyhedral partitions are introduced by Gromov in
\cite{Grom}. His {\em spaces of partitions}\/
\cite[Section~5]{Grom} are defined as the configuration spaces of
labelled binary trees $T_d$ of height $d$, with $2^d-1$ internal
nodes $N_d$ and $2^d$ external nodes $L_d$ (leaves of the tree
$T_d$). More explicitly a labelled binary tree $(T_d,
\{H_\nu\}_{\nu\in N_d})$ has an oriented hyperplane $H_\nu$
associated to each of the internal nodes $\nu\in N_d$ of $T_d$.
The left (respectively right) outgoing edge, emanating from
$\nu\in N_d$ is associated the positive half-space $H^+_\nu$
(respectively the negative half-space $H^-_\nu$) determined by
$H_\nu$.

Each of the leaves $\lambda\in L_d$ is the end point of the unique
maximal path $\pi_\lambda$ in the tree $T_d$. Each of the maximal
paths $\pi_\lambda$ is associated a polyhedral region $Q_\lambda$
defined as the intersection of all half-spaces associated to edges
of the path $\pi_\lambda$. The associated partition
$\{Q_\lambda\}_{\lambda\in L_d}$ depends continuously on the
chosen labels (hyperplanes) and defines an element of the
associated `space of partitions'.

These and related configuration spaces were used in \cite{Grom}
for the proof a general Borsuk-Ulam type theorem
($c_{\bullet}$-Corollary 5.3 on page 188) and utilized by Gromov
for his proof of the {\em Waist of the Sphere Theorem}.

\medskip
Very interesting `Voronoi polyhedral partitions' of measures were
recently introduced in \cite{Hub-Aro}. Far reaching results about
polyhedral equipartitions of measures were along these lines
obtained by Soberon \cite{Soberon}, Karasev \cite{Karasev}, and
Aronov and Hubard \cite{Hub-Aro}.

\subsection{Polynomial measure partitions theorems}

Theorem~\ref{thm:main-1}, being a relative of the Ham Sandwich
Theorem, has some standard consequences and extensions. One of
them is the Polynomial Ham Sandwich Theorem \cite{Sto-Tuk}, which
has recently found striking applications to some old problems of
discrete and computational geometry, see \cite{Gu-Ka},
\cite{So-Tao} and the references in these papers. These
breakthroughs have generated a lot of interest and enthusiasm, see
the reviews of J.~Matou\v{s}ek \cite{Jiri}, M.~Sharir
\cite{Sharir} and T.~Tao \cite{Tao-blog}.

\medskip Theorem \ref{thm:main-1}, as
well as other results about hyperplane measure partitions, have
immediate polynomial versions. They can be obtained by the usual
Veronese map $\mathbb{R}^d \hookrightarrow \mathbb{R}^D$, or some
of its variations. It remains to be seen if some of these
polynomial measure partition results can be used as a natural tool
which can replace the standard polynomial ham sandwich theorem in
some applications.

\section{Preliminaries, definitions, notation}

\subsection{Manifold complexes} \label{sec:methods}

\begin{defn}\label{def:nice_spaces}
A space $X$ is called a {\em manifold pre-complex} if it is either
a compact topological manifold (with or without boundary)  or if
it is obtained by gluing a compact topological manifold with
boundary to a manifold pre-complex via a continuous map of the
boundary.
\end{defn}

This definition appeared in \cite{Kreck} (Chapter~9) where
manifold pre-complexes are referred to as {\em nice spaces}. One
can {\em mutatis mutandis} modify the
Definition~\ref{def:nice_spaces} by allowing different kinds of
``manifolds''. For example $X$ is a pseudomanifold pre-complex
(orientable manifold pre-complex, complex manifold pre-complex,
etc.), if the constituent ``manifolds'' are pseudomanifolds
(orientable manifolds, complex manifolds, etc.).

\medskip
A manifold pre-complex should be seen as a straightforward
generalization of a (finite) $CW$-complex. However a $CW$-complex
has a natural filtration (and an associated rank function) so for
this reason we slightly modify the definition and introduce {\em
manifold complexes}.

\begin{defn}\label{def:strict_manifold_complex}
A space $X$ with a finite filtration
\begin{equation}\label{eqn:filtration-complex}
X_0\subset X_1\subset \ldots \subset X_{n-1}\subset X_n
\end{equation}
is called a {\em manifold complex} (orientable manifold complex,
etc.) if
\begin{itemize}
 \item[(1)] $X_0$ is a finite set of points ($0$-dimensional
manifold);
 \item[(2)]  For each $k \leq n,$ $X_k = X_{k-1}\cup_\phi Y_k$ where
             $Y_k$ is a compact $k$-dimensional manifold with
             boundary and $\phi : \partial Y_k \rightarrow
             X_{k-1}$ is a continuous map.
\end{itemize}
\end{defn}

As a variation on a theme we introduce manifold complexes with an
action of a finite group $G$.

\begin{defn}\label{def:G-manifold-complex} Let $G$ be a finite group.
A $G$-space $X$ which is also a manifold complex in the sense of
Definition~\ref{def:strict_manifold_complex} is called a
$G$-manifold complex if  $G$ preserves the filtration
(\ref{eqn:filtration-complex}).
\end{defn}

\subsection{Dihedral group $\mathbb{D}_8$}
\label{sec:notation}

For basic notation and standard facts about group actions the
reader is referred to \cite{tDieck87}. A representation space $V$
for a given $G$-representation $\rho : G\rightarrow GL(V)$ is also
referred to as a (real or complex) $G$-module. $S(V)$ is the unit
sphere in an orthogonal (or unitary) $G$-representation space $V$.
$X\ast Y$ is the join of two spaces $X$ and $Y$, and a standard
fact is that if $U$ and $V$ are two orthogonal $G$-modules,
$S(U\oplus V)$ and $S(U)\ast S(V)$ are isomorphic as $G$-spaces.

\medskip
$\mathbb{D}_8$ is the dihedral group of order $8$. $\Lambda =
\mathbb{Z}[\mathbb{D}_8]$ is the integral group ring of
$\mathbb{D}_8$ and $\Lambda x$ denotes the free, one-dimensional
$\Lambda$-module generated by $x$. In this paper $x$ is often a
fundamental class of an orientable $\mathbb{D}_8$-pseudomanifold
with boundary.

As the group of symmetries of the square $Q = \{(x,y)\in
\mathbb{R}^2 \mid 0\leq \vert x\vert, \vert y\vert \leq 1 \}$ the
dihedral group $\mathbb{D}_8$ has three distinguished involutions
$\alpha, \beta$ and $\gamma$ where
 \begin{equation}\label{eqn:action}
 \alpha(x,y)=(-x,y), \quad \beta(x,y)=(x,-y), \quad
 \gamma(x,y)=(y,x).
 \end{equation}
A standard presentation of $\mathbb{D}_8$ is,
$$\mathbb{D}_8 = \langle \alpha, \beta, \gamma \mid \alpha^2 =
\beta^2 = \gamma^2 = 1 ,\, \alpha\beta = \beta\alpha,\,
\alpha\gamma = \gamma\beta,\, \beta\gamma = \gamma\alpha
\rangle.$$

The real two-dimensional representation $ \rho : \mathbb{D}_8
\rightarrow O(2)$ arising from the action on the square is denoted
by $U$. Let $W := U\oplus\lambda$ where $\lambda$ is the
one-dimensional (real) $\mathbb{D}_8$-representation, such that
\begin{equation}\label{eqn:one-dim-rep}
\alpha \cdot z = - z, \quad \beta\cdot z = -z, \quad \gamma\cdot z
= z.
\end{equation}
Interpreting $\mathbb{D}_8$ as a Sylow $2$-subgroup of the
symmetric group $S_4$ we see that the $\mathbb{D}_8$-module $W$ is
isomorphic to the restriction of the reduced permutation
representation of $S_4$ to the dihedral group.

\begin{figure}[hbt]
\centering
\includegraphics[scale=0.55]{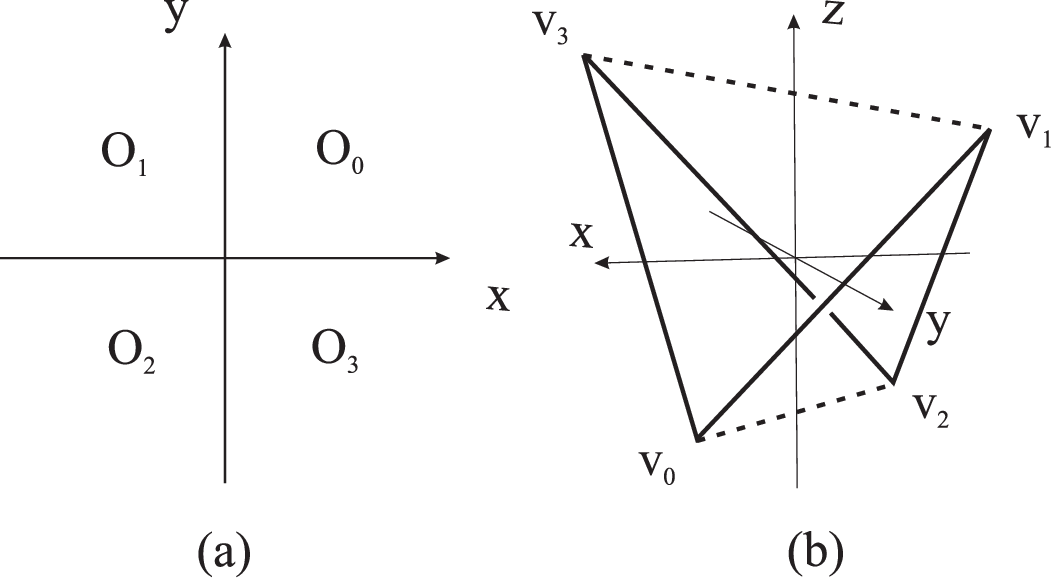}
\caption{$\mathbb{D}_8$-module $W$ as a permutation
representation.} \label{fig:tetra}
\end{figure}

More explicitly, as shown in Figure~\ref{fig:tetra}, the
permutations associated to the basic involutions $\alpha, \beta$
and $\gamma$ are:
\begin{equation}\label{eqn:permutations}
\alpha = \left (
\begin{array}{cccc}
0 & 1 & 2 & 3 \\
1 & 0 & 3 & 2
\end{array}
\right), \quad  \beta = \left (
\begin{array}{cccc}
0 & 1 & 2 & 3 \\
3 & 2 & 1 & 0
\end{array}
\right), \quad \gamma = \left (
\begin{array}{cccc}
0 & 1 & 2 & 3 \\
0 & 3 & 2 & 1
\end{array}
\right)
\end{equation}

\section{The topology of the equipartition
problem}\label{sec:topology}

The problem of deciding if for a given collection $\{\mu_1,\ldots,
\mu_j\}$ of measures in $\mathbb{R}^d$ there exists an
equipartition by an ordered pair $(H_0,H_1)$ of (oriented)
hyperplanes, can be reduced to a topological problem via the usual
{\em Configuration Space/Test Map Scheme}, see
\cite[Section~14.1]{Z04} or \cite[Section~2.3]{MVZ}. Here is a
brief outline of this construction.

The configuration space, or the space of all candidates for the
equipartition, is the space of all ordered pairs $(H_0,H_1)$ of
oriented hyperplanes. After a suitable compactification this space
can be identified as $S^d\times S^d$. The natural group of
symmetries for the equipartition problem is the dihedral group
$\mathbb{D}_8$ and the associated action on $S^d\times S^d$ is
given by formulas (\ref{eqn:action}).

The importance of the representation $W$ stems from the fact that
it naturally arises \cite[Section~2.3.3]{MVZ} as the associated
{\em Test Space} for a single (probabilistic) measure $\mu$.
Indeed for each ordered pair $(H_0,H_1)$ of oriented hyperplanes
there is an associated collection of hyperorthants $O_0, O_1, O_2,
O_3$ (Figure~\ref{fig:tetra}~(a)). Then $\mu(O_\nu),\,
\nu=0,1,2,3$ are naturally interpreted as the {\em barycentric
coordinates} of a point $v =\mu(O_0)v_0 + \mu(O_1)v_1 +
\mu(O_2)v_2 + \mu(O_3)v_3\in W$ (Figure~\ref{fig:tetra}~(b)) and
the action of $\mathbb{D}_8$ on $S^d\times S^d$ induces an action
on these barycentric coordinates which is precisely the action on
$W$ as a $\mathbb{D}_8$-module described in
Section~\ref{sec:notation}. Assume that the barycenter of the
tetrahedron is the origin, i.e.\ $(1/4)(v_0+v_1+v_2+v_3)=0\in W$.
The map
$$F_\mu : S^d\times S^d
\rightarrow W, \quad (H_0,H_1)\mapsto \mu(O_0)v_0 + \mu(O_1)v_1 +
\mu(O_2)v_2 + \mu(O_3)v_3$$ has the property that $(H_0,H_1)$ is
an equipartition for $\mu$ if and only if $(H_0,H_1)$ is a zero of
$F_\mu$. More generally $z:=(H_0,H_1)$ is an equipartition for the
collection $\{\mu_1,\ldots,\mu_j\}$ of probability measures if and
only if $z=(H_0,H_1)$ is a common zero of the associated test maps
$F_{\mu_j}$,
$$
F_{\mu_1}(z) = F_{\mu_2}(z) = \ldots = F_{\mu_j}(z)=0.
$$
Summarizing we have the following proposition.

\begin{prop}\label{prop:admissible}
A triple $(d,j,2)$ is admissible if each
$\mathbb{D}_8$-equivariant map
$$
F : S^d\times S^d \rightarrow W^{\oplus j}
$$
has a zero, or equivalently if there does not exist a
$\mathbb{D}_8$-equivariant map
\begin{equation}\label{eqn:equi-map}
f : S^d\times S^d \rightarrow S(W^{\oplus j}).
\end{equation}
\end{prop}

\section{Standard admissible filtration of $S^n\times S^n$}
\label{sec:standard}

In order to prove the non-existence of an equivariant map
(\ref{eqn:equi-map}) we apply the equivariant obstruction theory
in the form outlined in Section~\ref{sec:appendix}. The first step
is a construction of an appropriate  filtration on $S^n\times S^n$
which is admissible in the sense of
Definition~\ref{defn:adm-filtration}. Less formally, we turn
$S^n\times S^n$ into a $\mathbb{D}_8$-manifold complex in the
sense of Definition~\ref{def:G-manifold-complex} by allowing
orientable $\mathbb{D}_8$-manifolds with corners and mild
singularities.

\medskip
For $i=1,\ldots,n+1$ define $\pi_i : S^n\times S^n \rightarrow
\mathbb{R}^2, \, \pi_i(x,y) := (x_i,y_i)$ as the restriction of
the obvious projection map. The maps $\pi_i$ are clearly
$\mathbb{D}_8$-equivariant and the images  ${\rm Image}(\pi_i)=
\{(x_i,y_i)\mid -1\leq x_i,y_i\leq +1\} =: Q_i$ are squares which
are here referred to as ``$\mathbb{D}_8$-screens'',
Figure~\ref{fig:screens-1}.  The screens $Q_i$ admit a
$\mathbb{D}_8$-invariant triangulation which is the starting point
for the construction of an admissible filtration on $S^n\times
S^n$.

\begin{figure}[hbt]
\centering
\includegraphics[scale=0.50]{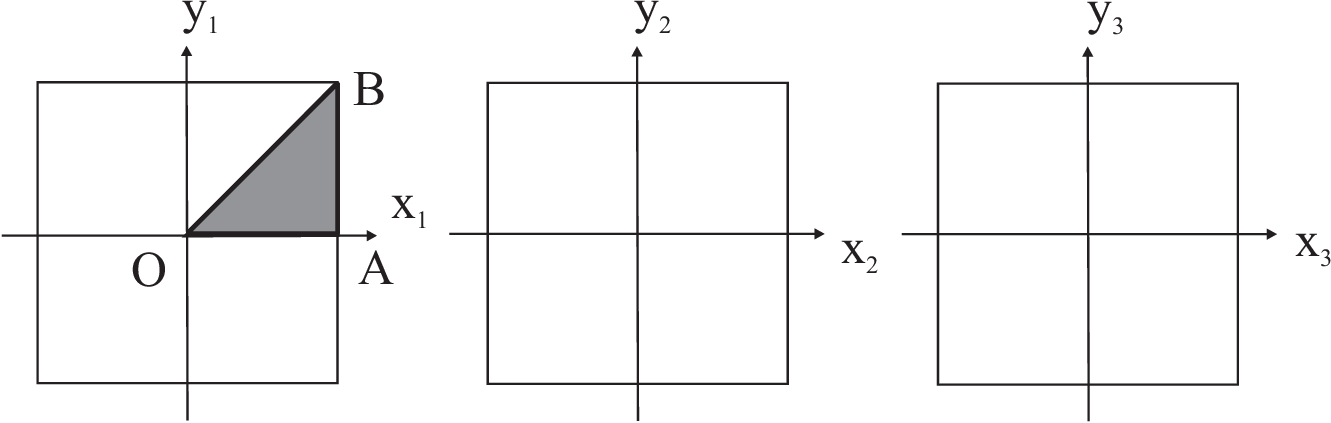}
\caption{$\mathbb{D}_8$-screens for $S^2\times S^2\subset
\mathbb{R}^3\times \mathbb{R}^3$.} \label{fig:screens-1}
\end{figure}
The filtration can be informally described as follows. The
manifold $S^n\times S^n$ is fibered over the first screen $Q_1$
with a generic fiber homeomorphic to $S^{n-1}\times S^{n-1}$. The
fiber $S^{n-1}\times S^{n-1}$ itself is fibered over the second
screen $Q_2$ with $S^{n-2}\times S^{n-2}$ as a generic fiber, etc.
This priority order of screens together with their minimal
$\mathbb{D}_8$-invariant triangulations are used to define a
version of ``lexicographic filtration'' of $S^n\times S^n$. For
the intended application of the obstruction theory methods from
Section~\ref{sec:appendix} it will be sufficient to give a precise
description only for the first three terms of this filtration.

\medskip
The $\pi_1$-preimage $X := \pi_1^{-1}(\Delta_{OAB})$ of the
triangle $\Delta_{OAB}\subset Q_1$  is (the closure of) a
fundamental domain of the $\mathbb{D}_8$-action on $S^n\times
S^n$. It is described by the inequalities $0\leq y_1\leq x_1\leq
1$ and as a subset of $S^n\times S^n$ it is an orientable manifold
with boundary. This manifold has corners and possibly
singularities of high codimension, however the associated
fundamental class $x\in H_{2n}(X,\partial X)$ is well defined. The
geometric boundary of $X$ is,
$$
\partial X = \pi_1^{-1}(\partial\Delta_{OAB}) = X_0'\cup X_1' \cup
X_2' = \pi_1^{-1}(OA)\cup \pi_1^{-1}(OB)\cup \pi_1^{-1}(AB).
$$
If $n\geq 2$ homologically significant are only $X_0'$ and $X_1'$
since $X_2'$ has codimension $n$ in $S^n\times S^n$ and they
contribute to the (homological) boundary evaluated in dimension
$2n-1$. Since ${\rm Stab}(X_0')=\langle\beta\rangle$ we subdivide
and define $X_0=X_0'\cap\{y_2\geq 0\}$ as the associated
fundamental domain. Similarly, ${\rm
Stab}(X_1')=\langle\gamma\rangle$ and $X_1=X_1'\cap \{y_2\leq
x_2\}$, cf.\ Figure~\ref{fig:screens-2} and the subsequent
tree-like diagram.

\vspace{-2mm}
\begin{figure}[hbt]
\centering
\includegraphics[scale=0.50]{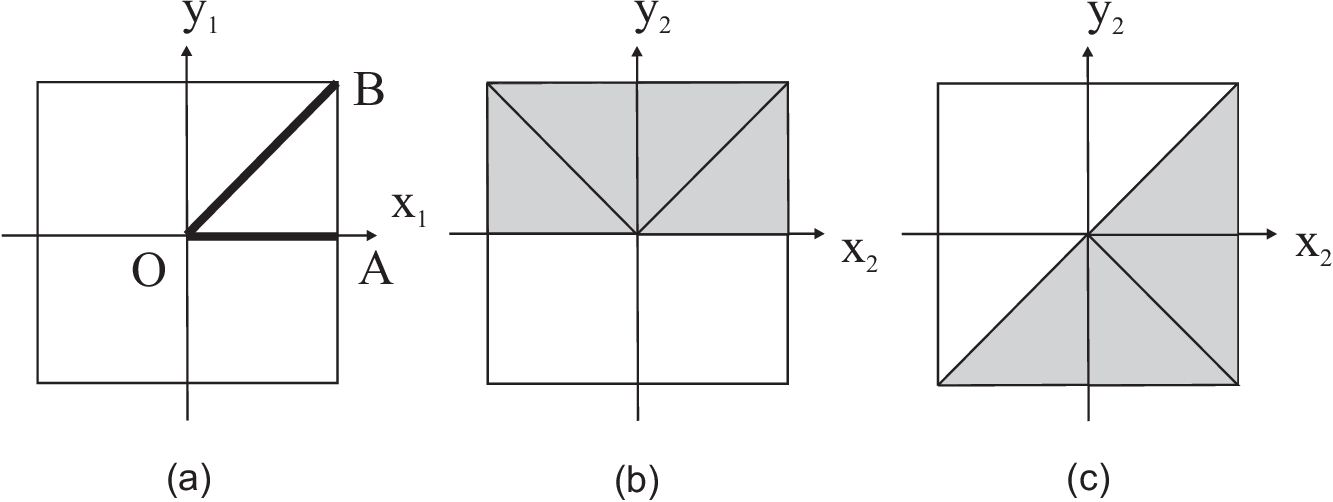}
\vspace{-2mm} \caption{Admissible filtration for $S^n\times S^n$.}
\label{fig:screens-2}
\end{figure}
Assuming that $n\geq 8$, we continue the ``subdivide and take the
boundary''-procedure, focusing only on the homologically
significant part of the boundary.
\[
\begin{array}{ccccl}
\partial X_0 & = & Z_0'\cup T_0' & = & \{x_1=y_1=0, y_2\geq 0\}\cup \{x_1\geq 0, y_1=y_2=0\} \\
\partial X_1 & = & Z_0''\cup T_1' & = & \{x_1=y_1=0, y_2\leq x_2\}\cup \{x_1=y_1, x_1\geq 0, x_2=y_2\}
\end{array}
\]
The sets $Z_0'$ and $Z_0''$ can be further subdivided as follows
(Figure~\ref{fig:screens-2})
\[
\begin{array}{rcl}
Z_0' & = & Z_0\cup \alpha Z_0\cup \gamma Z_0 \cup \alpha\gamma Z_0\\
Z_0''& = & Z_0\cup \beta Z_0\cup \beta\gamma Z_0 \cup
\alpha\beta\gamma Z_0
\end{array}
\]
where $Z_0 := \{x_1=y_1=0, 0\leq y_2\leq x_2\leq 1\}$. Finally,
$T_0' = T_0\cup \beta T_0$ and $T_1' = T_1\cup \gamma T_1$ where
$T_0 = T_0'\cap \{0\leq y_3\leq 1\}$ and $T_1 = T_1'\cap \{y_3\leq
x_3\}$.

\vspace{-2mm} {\small
\begin{equation*}
\xymatrix@C=0.01pc@R=2pc{ &&&& *+[F]\txt{$0\leq x_1\leq 1$\\$0\leq
y_1\leq 1$\\$y_1\leq x_1$}
\ar[dll]_\partial \ar[drr]^\partial &&&& \\
&&[y_1=0]\ar[dl]\ar@{.>}[dr] \ar@{}[d]|{\langle\beta\rangle} & &&
&
[x_1=y_1]\ar[dr]\ar@{.>}[dl]\ar@{}[d]|{\langle\gamma\rangle}&&\\
&[0\leq y_2]\ar[dl]_{\partial}\ar[dr]^{\partial}  & &[y_2\leq 0]
&& [x_2\leq y_2]  &
&[y_2 \leq x_2]\ar[dl]_{\partial}\ar[dr]^{\partial} &\\
*+[F]\txt{$x_1= 0$\\$y_1= 0$\\$0\leq y_2$} && *+[F--]\txt{$y_1=
0$\\$y_2= 0$\\$0\leq x_1$}  &&&& *+[F]\txt{$x_1= 0$\\$y_1=
0$\\$y_2\leq x_2$} && *+[F--]\txt{$x_1=y_1$\\$x_2=y_2$\\$0\leq
x_1$}}
\end{equation*}
}
\subsection{The fundamental domain - geometric boundary procedure}
As a summary of the construction we observe that the ``take the
boundary, then subdivide''-procedure produces an admissible
filtration
\begin{equation}\label{eqn:filtration}
S^n\times S^n = F_{2n} \supset F_{2n-1} \supset
F_{2n-2}\supset\ldots
\end{equation}
where $F_{2n-1}:=\bigcup_{g\in \mathbb{D}_8}~g(\partial X)$ and
$F_{2n-2}:= \bigcup_{g\in \mathbb{D}_8}~g(Z_0\cup T_0\cup T_1\cup
X_2)$.

For the intended application the explicit description
(\ref{eqn:filtration}) is sufficient. However, it is clear how one
can continue the construction by using screens of higher order.
Observe that the `tree of fundamental domains' is formally
generated by two types of branching while the root of the tree is
a fundamental domain of the manifold.

\section{Fragment of the chain complex for $S^n\times S^n$}
\label{sec:fragment}

The sets $X, X_0, X_1, Z_0, T_0, T_1$ described in the previous
section are connected manifolds with boundary (with corners and
possibly with mild singularities in codimension $\geq 2$). They
all can be oriented in which case the corresponding fundamental
classes are denoted by $x, x_0, x_1, z_0, t_0, t_1$. These classes
are naturally interpreted as the generators of
$\mathbb{D}_8$-modules $H_k(F_k,F_{k-1}; \mathbb{Z})$ for
$k=2n,2n-1,2n-2$.

\medskip
The orientation character of the $\mathbb{D}_8$-manifold
$S^n\times S^n$ is given by
\begin{equation}\label{eqn:or-character}
(\alpha, \beta, \gamma) = ((-1)^{n-1}, (-1)^{n-1}, (-1)^n).
\end{equation}
From (\ref{eqn:or-character}) and the analysis of geometric
boundaries given in Section~\ref{sec:standard} one deduces the
following relations:
\begin{equation}\label{eqn:boundary}
\begin{array}{ccl}
\partial x & = & (1+(-1)^n\beta)x_0 + (1+(-1)^{n-1}\gamma)x_1 \\
\partial x_0 & = & (1 + (-1)^n\alpha + (-1)^{n-1}\gamma -\alpha\gamma)z_0 + (1+(-1)^{n-1}\beta)t_0 \\
\partial x_1 & = & -(1+(-1)^n\beta -\beta\gamma + (-1)^{n-1}\alpha\beta\gamma)z_0
+ (1+(-1)^n\gamma)t_1
\end{array}
\end{equation}
The top dimensional fragment of the associated chain complex is,
\begin{equation}\label{eqn:chain-complex}
\xymatrix{ \Lambda x \ar[r]^-{B} & \Lambda x_0 \oplus  \Lambda x_1
\ar[r]^-{A} & \Lambda z_0\oplus \Lambda t_0\oplus \Lambda t_1
\ar[r] & \ldots }
\end{equation}
The boundary homomorphisms are described by (\ref{eqn:boundary})
so for example if $n$ is even,
\begin{equation}\label{eqn:matrice}
A^t = \left [
\begin{array}{ccc}
1 + \alpha - \gamma -\alpha\gamma & 1-\beta & 0 \\
-(1+\beta -\beta\gamma -\alpha\beta\gamma) & 0 & 1+\gamma
\end{array}
\right] \, B = \left [
\begin{array}{c}
1+\beta  \\ 1-\gamma
\end{array}
\right]
\end{equation}

\begin{rem}
The reader may use the $\mathbb{D}_8$-screens, introduced in
Section~\ref{subsec:admissible} and depicted in
Figures~\ref{fig:screens-1} and \ref{fig:screens-2}, as a useful
bookkeeping device for checking the formulas (\ref{eqn:boundary}).
For example the term $(1 + (-1)^n\alpha + (-1)^{n-1}\gamma
-\alpha\gamma)z_0$, in the middle equation, corresponds to the
decomposition of the shaded rectangle, shown in
Figure~\ref{fig:screens-2}-(b), into four triangles. Being in the
second screen, this rectangle corresponds to a region in
$S^{n-1}\times S^{n-1}$ whose orientation, following
(\ref{eqn:or-character}), transforms by the rule $(\alpha, \beta,
\gamma) = (-1)^n, (-1)^n, (-1)^{n-1}$. This explains the sign in
$(-1)^n\alpha$, etc.
\end{rem}

\section{Evaluation of the cohomology group
$\mathcal{H}^{2n-1}_{\mathbb{D}_8}(S^n\times S^n;\mathcal{Z})$}
\label{sec:evaluation}

The first obstruction to the existence of an equivariant map
(\ref{eqn:equi-map}) lies in the group
$\mathcal{H}^{2n-1}_{\mathbb{D}_8}(S^n\times S^n;\mathcal{Z})$
where $\mathcal{Z}\cong H_{3j-1}(S(W^{\oplus j}); \mathbb{Z})$ is
the orientation character of the sphere $S(W^{\oplus j})$. We
remind the reader that these groups are defined
(Section~\ref{sec:appendix}) as functors of $\mathbb{D}_8$-spaces
with admissible filtrations. Also note that the condition
$2d-3j=1$ (Section~\ref{sec:new}) allows us to assume that $j$ is
an odd integer.

\begin{prop}\label{prop:spectral}
Let $\mathcal{M}=\mathcal{Z}$ be the orientation character of the
sphere $S(W^{\oplus j})$ where $j$ is an odd integer. Then,
\begin{equation}\label{eqn:z4-cohomology}
{\mathcal{H}}^{2n-1}_{\mathbb{D}_8}(S^n\times
S^n;\mathcal{\mathcal{Z}}) \cong \left\{ \begin{array}{cl}
         \mathbb{Z}/4 & \mbox{if\, $n$ is even};\\
        \mathbb{Z}/2 \oplus \mathbb{Z}/2 & \mbox{if\, $n$ is odd}.\end{array}
        \right.
\end{equation}
Moreover a generator of $\mathbb{Z}/4$ can be described as the
cocycle
\begin{equation}\label{eqn:cocycle}
\phi : \Lambda x_0\oplus \Lambda x_1 \rightarrow \mathcal{Z},
\quad \phi(x_0)=1, \phi(x_1)=-1.
\end{equation}
\end{prop}
\begin{proof}
The orientation character $\mathcal{Z}$ of the
$\mathbb{D}_8$-sphere $S(W)$ is easily determined from the signs
of permutations (\ref{eqn:permutations}) and reads as follows
\begin{equation}\label{eqn:karakter}
(\alpha,\beta,\gamma) = (+1,+1,-1).
\end{equation}
Since $j$ is an odd integer the same answer is obtained for the
orientation character of the sphere $S(W^{\oplus j})$. Assume that
$n$ is even. By applying the functor ${\rm Hom}(\cdot ,
\mathcal{Z})$ to the chain complex (\ref{eqn:chain-complex}) we
obtain the complex
\begin{equation}\label{eqn:cohomology-calc}
\xymatrix{ \mathbb{Z}  & \mathbb{Z} \oplus  \mathbb{Z}
\ar[l]_-{B_1} & \mathbb{Z}\oplus \mathbb{Z}\oplus \mathbb{Z}
\ar[l]_-{A_1} }
\end{equation}
where,
\[
A_1 = \left [
\begin{array}{ccc}
4 & 0 & 0 \\
-4 & 0 & 0
\end{array}
\right] \quad B_1^t = \left [
\begin{array}{c}
2  \\ 2
\end{array}
\right]
\]
From here we deduce that
\begin{equation}
\mathcal{H}^{2n-1}_{\mathbb{D}_8}(S^n\times S^n;\mathcal{Z}) \cong
\mathbb{Z}/4
\end{equation}
where a generator is the cocycle $\phi : \Lambda x_0\oplus \Lambda
x_1 \rightarrow \mathcal{Z}, \, \phi(x_0)=1, \phi(x_1)=-1$. The
second half of (\ref{eqn:z4-cohomology}) is established by a
similar calculation.
\end{proof}

\section{Evaluation of the obstruction}
\label{sec:obstr-evaluation}

In light of Propositions~\ref{prop:admissible} and
\ref{prop:spectral} (isomorphisms (\ref{eqn:z4-cohomology})) our
primary concern are admissible triples $(d,j,2)$ such that $d$ is
an even integer; this is precisely the case when the obstruction
group is $\mathbb{Z}/4$. Hence, throughout this section we assume
that $(d,j)=(6k+2,4k+1)$ for some integer $k\geq 0$.

\medskip
As emphasized in Section~\ref{subsec:heuristic} the evaluation of
the obstruction class $\theta$ can be linked with the computation
of mapping degrees of carefully chosen maps. We begin with a map
which plays an auxiliary role in these calculations.

\subsection{The degree of the multiplication of monic polynomials}

Let $\mathcal{P}_m := \{p(x) = a_0 + a_1x+\ldots +a_{m-1}x^{m-1}+
a_mx^m \mid a_i\in \mathbb{R}\}$ be the vector space of
polynomials of degree at most $m$ with coefficients in the field
of real numbers. Let
$$
\mu_{m,n} : \mathcal{P}_m \times \mathcal{P}_n \rightarrow
\mathcal{P}_{m+n}
$$
be the multiplication of polynomials. We focus our attention on
the (affine) space $\mathcal{P}_m^0 := \{p(x) = a_0 + a_1x+\ldots
+a_{m-1}x^{m-1}+ x^m \mid  a_i\in \mathbb{R}\}$ of monic
polynomials of degree $m$ and the associated multiplication map
\begin{equation}\label{eqn:monic-multiplication}
\mu_{m,n}^0 : \mathcal{P}_m^0 \times \mathcal{P}_n^0 \rightarrow
\mathcal{P}_{m+n}^0
\end{equation}

Our objective is to evaluate the mapping degree of the map
$\mu_{m,n}^0$, say as the algebraic count of the number of points
in the pre-image $f^{-1}(z)$ of a regular point $z\in
\mathcal{P}_{m+n}^0$. We begin with a preliminary proposition
which guarantees that the degree is well defined.

\begin{prop}\label{prop:proper}
The multiplication {\em (\ref{eqn:monic-multiplication})} of monic
polynomials is a proper map of manifolds.
\end{prop}

\begin{proof} Assume that $A\subset \mathcal{P}_m^0$ and $B\subset
\mathcal{P}_n^0$ are sets of polynomials such that $A\cdot B
:=\{p\cdot q \mid p\in A, q\in B\}$ is bounded as a set of
polynomials in $\mathcal{P}_{m+n}^0\cong \mathbb{R}^{m+n}$. We
want to conclude that both $A$ and $B$ individually a bounded sets
of polynomials. This is easily deduced from the following claim.

\medskip\noindent
{\bf Claim:} If $A\subset \mathcal{P}_n^0$ is bounded set of
polynomials then the set $Root(A) := \{z\in \mathbb{C} \mid p(z)=0
\mbox{ {\rm for some} } p\in A \}$ is also bounded. Conversely, if
$Root(A)$ is a bounded, $A$ is also a bounded set of polynomials.

\medskip\noindent
{\bf Proof of the Claim:} The implication $\Leftarrow$ follows
from Vi\`{e}te's formulas, while the opposite implication
$\Rightarrow$ follows from the inequality $\vert \lambda\vert \leq
\mbox{\rm Max}\{1, \sum_{j=0}^{n-1} \vert a_j\vert\}$, where
$\lambda$ is a root of a polynomial with coefficients $a_j$.
\end{proof}

The next step needed for computation of the mapping degree of the
map (\ref{eqn:monic-multiplication}) is the evaluation of the
differential $d\mu^0_{m,n}$. The tangent space
$T_p(\mathcal{P}_m^0)$ at the monic polynomial $p\in
\mathcal{P}_m^0$ is naturally identified with the space
$\mathcal{P}_{m-1}$ of all polynomials of degree at most $m-1$.

\begin{lem} Given monic polynomials $p\in \mathcal{P}_m^0$ and
$q\in \mathcal{P}_n^0$ and the polynomials $u\in
\mathcal{P}_{m-1}, v\in \mathcal{P}_{n-1}$, playing the role of
the associated tangent vectors, the differential $d\mu^0_{m,n} =
d\mu$ is evaluated by the formula
\begin{equation}\label{eqn:dif}
d\mu_{(p,q)}(u,v) = \frac{d}{dt}(p+tu)(q+tv)_{\vert_{t=0}} = pv +
uq.
\end{equation}
\end{lem}

Let us determine the matrix of the map $d\mu^0_{(p,q)}$ in
suitable bases of the associated tangent spaces
$T_{(p,q)}(\mathcal{P}_m^0\times \mathcal{P}_n^0)\cong
\mathcal{P}_{m-1} \times \mathcal{P}_{n-1}$ and
$T_{pq}(\mathcal{P}_{m+n}^0)\cong \mathcal{P}_{m+n-1}$. A
canonical choice of basis for $\mathcal{P}_{m-1}^0$ is $u_0 = 1,
u_1 = x , \ldots , u_{m-1} = x^{m-1}$ with similar choices $v_0 =
1, v_1=x, \ldots , v_{n-1} = x^{n-1}$ and $w_0=1, w_1=x,\ldots ,
w_{p+q-1}=x^{m+n-1}$ for $\mathcal{P}_{n-1}^0$ and
$\mathcal{P}_{m+n-1}^0$ respectively. Formula (\ref{eqn:dif})
applied to this basis gives
$$
d\mu_{(p,q)}(0, v^j) = d\mu_{(p,q)}(0, x^j) = x^jp(x) \qquad
d\mu_{(p,q)}(u_i, 0) = d\mu_{(p,q)}(x^i, 0) = x^iq(x).
$$
We conclude from here that the determinant of this matrix is equal
to the {\em resultant} (\ref{eqn:resultant}) of two polynomials!
 {\small
\begin{equation}\label{eqn:resultant}
\begin{array}{ccc}
\mathcal{R}(p,q) & = & Det
\begin{bmatrix}
  a_0 & a_1 & \ldots & a_{m-1} & 0 & 0 & \ldots & 0 \\
  0   & a_0 & a_1 & \ldots & a_{m-1}  & 0 & \ldots & 0\\
  \vdots & \vdots & \ddots & \ddots & \vdots & \vdots & \ddots &
  \vdots\\
  0 & 0 & \ldots & 0 & a_0 & a_1 & \ldots & a_{m-1}\\
  b_0 & b_1 & \ldots & b_{n-1} & 0 & 0 & \ldots & 0 \\
  0   & b_0 & b_1 & \ldots & b_{n-1}  & 0 & \ldots & 0\\
  \vdots & \vdots & \ddots & \ddots & \vdots & \vdots & \ddots &
  \vdots\\
  0 & 0 & \ldots & 0 & b_0 & b_1 & \ldots & b_{n-1}
  \end{bmatrix}
  \end{array}
\end{equation}
 }
In particular we can use classical formulas for
$\mathcal{R}(p,q)$, see \cite[Chapter~12]{GKZ}, among them the
formula
\begin{equation}\label{eqn:resultant-formula}
\mathcal{R}(p,q) = \prod_{i, j} (\alpha_i - \beta_j)
\end{equation}
where $\alpha_i$ are roots of $p$ and $\beta_j$ are roots of $q$
respectively, counted with the appropriate multiplicities.

\begin{lem}\label{lema:znak} Suppose that $p(x) = a_0 +
a_1x+\ldots + a_{m-1}x^{m-1} + x^m$ and $p(x) = b_0 + b_1x+\ldots
+ b_{n-1}x^{n-1} + x^n$ are two polynomials with real coefficients
such that the corresponding roots $\alpha_1,\ldots ,\alpha_m$ and
$\beta_1,\ldots, \beta_n$ are all distinct and non-real,
$\{\alpha_i\}_{i=1}^m\cup \{\beta_j\}_{j=1}^n\subset
\mathbb{C}\setminus \mathbb{R}$. Then the resultant of polynomials
$p$ and $q$ is real and positive, $\mathcal{R}(p,q) > 0$.
\end{lem}

\begin{proof} By assumption all roots of $p(x)$ (respectively $q(x)$) can
be divided in conjugate pairs $\alpha, \overline{\alpha}$
(respectively $\beta, \overline{\beta}$). These two pairs
contribute to the product (\ref{eqn:resultant-formula}) the factor
$$
(\alpha - \beta)(\overline{\alpha} - \overline{\beta})(\alpha -
\overline{\beta})(\overline{\alpha} - \beta) = A\overline{A}
B\overline{B} > 0.
$$
\end{proof}

\begin{prop}\label{prop:degree-multiplication}
Suppose that $m=2k$ and $n=2l$ are even integers. The degree of
the map $\mu^0_{m,n} : \mathcal{P}_m^0\times
\mathcal{P}_n^0\rightarrow \mathcal{P}_{m+n}^0$ is
\begin{equation}\label{eqn:binomial}
\mbox{\rm deg}(\mu^0_{m,n})=  {{k+l}\choose{k}}.
\end{equation}
\end{prop}

\begin{proof}
We compute the degree $\mbox{\rm deg}(\mu^0_{m,n})$ by an
algebraic count of the number of points in the pre-image
$(\mu^0_{m,n})^{-1}(\rho)$ where the polynomial $\rho\in
\mathcal{P}_{m+n}^0$ is a regular value of the map $\mu^0_{m,n}$.

Assume that $\rho = \rho_1\rho_2\ldots \rho_{k+l}$ is a product of
pairwise distinct, irreducible, quadratic (monic) polynomials
$\rho_i$. Equivalently $\rho$ does not have real roots and all its
roots are pairwise distinct. Note that such a polynomial can be
easily constructed by prescribing its roots, for example it can be
found in any neighborhood of the polynomial $x^{m+n}$. The inverse
image $(\mu^0_{m,n})^{-1}(\rho)$ is,
$$
(\mu^0_{m,n})^{-1}(\rho) = \{(p,q)\in \mathcal{P}_m^0\times
\mathcal{P}_n^0 \mid p\cdot q=\rho\}.
$$
It follows from Lemma~\ref{lema:znak} that $\mathcal{R}(p,q)>0$
for each pair of polynomials in the inverse image
$(\mu^0_{m,n})^{-1}(\rho)$. In particular $\rho$ is a regular
value of the map $\mu^0_{m,n}$ and each pair $(p,q)$ such that
$p\cdot q=\rho$ contributes $+1$ to the degree. From here we
deduce that
$$
\mbox{\rm deg}(\mu^0_{m,n})= {{k+l}\choose{k}}
$$
\end{proof}

The multiplication map $\mu_{m,n}(p,q)=p\cdot q$ is non-degenerate
in the sense that $\mu_{m,n}(p,q)=0$ implies that either $p=0$ or
$q=0$. As a consequence it induces a map
$$
\mu_{m,n}' : \mathcal{P}_m'\times \mathcal{P}_n'\rightarrow
\mathcal{P}_{m+n}'
$$
where $\mathcal{P}_m':= \mathcal{P}_m\setminus\{0\}$ is the set of
non-zero polynomials.

\begin{prop}\label{prop:degree-multiplication-bis}
Suppose that $m=2k$ and $n=2l$ are even integers. The degree of
the map $\mu_{m,n}' : \mathcal{P}_m'\times
\mathcal{P}_n'\rightarrow \mathcal{P}_{m+n}'$ is
\begin{equation}\label{eqn:binomial-bis}
\mbox{\rm deg}(\mu_{m,n}')= \pm\, 2\cdot{{k+l}\choose{k}}.
\end{equation}
\end{prop}
\begin{proof}
Let $\mathcal{P}_d^0 := \{a_0+\ldots +a_{d}t^{d}\in \mathcal{P}_d
\mid a_d=1\}$ be the hyperplane of monic polynomials, viewed as
the tangent space at $t^{d}$ to the euclidean sphere $S^{d-1} =
\{p\in\mathcal{P}_d\mid \|p \|=1  \}$ in the space of polynomials
$\mathcal{P}_d$. The degree $\mbox{\rm deg}(\mu_{m,n}')$ can be
evaluated as the degree of the map
\begin{equation}
\hat\mu=\hat\mu_{m,n} : S^m \times S^n \rightarrow S^{m+n}
\end{equation}
where $\hat\mu(p,q) = (p\cdot q)/\| p\cdot q \|$ is the
multiplication of polynomials composed with the radial projection.
The radial projection also induces a bijection of
$\mathcal{P}_d^0$ with the open hemisphere $S^d_+
:=\{p\in\mathcal{P}_d\mid \|p \|=1, \, a_d>0\}$. Similarly
$-\mathcal{P}_d^0$ is radially projected on the negative (open)
hemisphere $S^d_-$.

Following the proof of
Proposition~\ref{prop:degree-multiplication}, let us choose for a
regular value $\rho$ of the map $\mu^0_{m,n}$ a polynomial very
close to the north pole $t^{m+n}$ of the sphere $S^{m+n}$. We
observe that each decomposition $p\cdot q = \rho$, contributing to
the degree of $\mu^0_{m,n}$, defines two points, $(p,q)$ and
$(-p,-q)$ in the preimage $\hat\mu^{-1}(\rho)$. Since the map
$(x,y)\mapsto (-x,-y)$ preserves the orientation of the manifold
$S^m\times S^n$, we observe that both $(p,q)$ and $(-p,-q)$
contribute to the degree of $\hat\mu_{m,n}$ with the same sign
which is independent of the choice of the decomposition $p\cdot q
= \rho$. As a consequence, $\mbox{\rm deg}(\hat\mu_{m,n}) = 2\cdot
\mbox{\rm deg}(\mu^0_{m,n})$ which in light of
(\ref{eqn:binomial}) completes the proof of the proposition.
\end{proof}

\begin{cor}\label{cor:degree} Suppose that
$f : S^m\times S^m \rightarrow S(\lambda^{\oplus (2m+1)})$ is a
$\mathbb{D}_8$-equivariant map where $m=2k$ is an even integer.
Then
$$
\mbox{\rm deg}(f) \equiv 4{{2k-1}\choose{k-1}} \quad (\mbox{\rm
mod}\, 8).
$$
\end{cor}
\begin{proof}
By Corollary~\ref{cor:prvi} (Section~\ref{sec:appendix-I}) it is
sufficient to exhibit a single map with the indicated degree. The
multiplication map $\mu_{m,m} : \mathcal{P}_m\times
\mathcal{P}_m\rightarrow \mathcal{P}_{2m}$ is
$\mathbb{D}_8$-equivariant if $\mathcal{P}_{2m}\cong
\lambda^{\oplus (2m+1)}$ as $\mathbb{D}_8$-modules. Hence, the
result follows from
Proposition~\ref{prop:degree-multiplication-bis}.
\end{proof}

\begin{exmp}
The $\mathbb{D}_8$-equivariant map $\mu : S^2\times S^2\rightarrow
S^4$ associated to the multiplication of quadratic polynomials
$\mu_{2,2}' : \mathcal{P}_2'\times \mathcal{P}_2'\rightarrow
\mathcal{P}_{4}'$ has degree $\pm 4$, consequently the degree of
any $\mathbb{D}_8$-equivariant map is an integer of the form
$8k+4$.
\end{exmp}

\subsection{Evaluation of the obstruction class $\theta$}

The following proposition is a key result for evaluation of the
primary obstruction $\theta\in \mathbb{Z}/4$ to the existence of a
$\mathbb{D}_8$-equivariant map (\ref{eqn:equi-map}).

\begin{prop}\label{prop:exists} Suppose that
\begin{equation}\label{eqn:exists}
\phi : S^{d-1}\times S^{d-1}\rightarrow S(W^{\oplus j})
\end{equation}
is a $\mathbb{D}_8$-equivariant map where $(d,j)=(6k+2,4k+1)$ for
some integer $k\geq 0$. Then,
\begin{equation}\label{eqn:degree-fundamental}
\mbox{\rm deg}(\phi) \equiv 4{{2k-1}\choose{k-1}} \quad (\mbox{\rm
mod}\, 8).
\end{equation}
\end{prop}

\begin{proof} By Corollary~\ref{cor:drugi} it is
sufficient to exhibit a particular $\mathbb{D}_8$-equivariant map
(\ref{eqn:exists}) which satisfies the congruence
(\ref{eqn:degree-fundamental}). Let us construct a
$\mathbb{D}_8$-equivariant map
\begin{equation}\label{eqn:we-will}
\Phi : \mathbb{R}^d\times \mathbb{R}^d \rightarrow W^{\oplus j}
\end{equation}
with the property that if $\Phi(p,q)=0$ then either $p=0$ or
$q=0$. By decomposing the  real $\mathbb{D}_8$-modules
$\mathbb{R}^d\times \mathbb{R}^d\cong U^{\oplus d}\cong U^{\oplus
j}\oplus U^{\oplus (d-j)}$ and $W^{\oplus j}\cong U^{\oplus
j}\oplus \lambda^{\oplus j}$ (Section~\ref{sec:notation}) we
observe that it is sufficient to construct a
$\mathbb{D}_8$-equivariant map
\begin{equation}\label{eqn:we-will-bis}
\Phi' : \mathbb{R}^{d-j}\times \mathbb{R}^{d-j} \rightarrow
\lambda^{\oplus j}
\end{equation}
which is non-degenerate in the sense that if $\Phi'(p,q)=0$ then
either $p=0$ or $q=0$. Let $m:=2k$ so
$(d,j)=(3m+2,2m+1)=(6k+2,4k+1)$. Identify $\mathbb{R}^{d-j}$ with
the space of real polynomials of degree less or equal to $d-j-1=m$
and $\lambda^{\oplus j}$ with the space of real polynomials of
degree $\leq j-1 = 2m$. Then the multiplication of polynomials
defines a non-degenerate (symmetric) bilinear form
$\Phi'(p,q)=p\cdot q$ which is an example of a
$\mathbb{D}_8$-equivariant map (\ref{eqn:we-will-bis}) with the
desired properties.

Summarizing, using the identifications $\mathbb{R}^d \cong
\mathcal{P}_{2m}\oplus \mathcal{P}_m$ and $W^{\oplus j}\cong
(\mathcal{P}_{2m})^{\oplus 3}$ with the corresponding vector
spaces of polynomials, we observe that the map $\Phi$ has the
following explicit form,
\begin{equation}\label{eqn:phi-explicit}
\Phi(p,q) = \Phi(p',p'';q',q'') = (p',q',p''q'')
\end{equation}
where $p''q'' = \mu(p'',q'')$ is the polynomial multiplication.
The degree ${\rm deg}(\Phi)$ can be calculated again, as in the
proof of Proposition~\ref{prop:degree-multiplication-bis}, by the
reduction to the case of monic polynomials $p''$ and $q''$. By
choosing the regular value of the map (\ref{eqn:phi-explicit}) in
the form $(0,0,\rho)$, where $\rho$ is a regular value for the
multiplication of monic polynomials we observe that
Proposition~\ref{prop:exists} is an immediate consequence of
Proposition~\ref{prop:degree-multiplication}.
\end{proof}

\medskip
The following proposition is the central result of this section
and the ultimate goal of all earlier computations. Note that the
manifold $S^{d-1}\times S^{d-1}$ (Proposition~\ref{prop:exists})
can be naturally identified as a subset of the
$(2d-2)$-dimensional element $F_{2d-2}$ of the filtration
(\ref{eqn:filtration}) via the identification $S^{d-1}\times
S^{d-1} = (S^{d}\times S^{d})\cap\{x_1=y_1=0\}= \bigcup_{g\in
\mathbb{D}_8}~g(Z_0)$.

\begin{prop}\label{prop:obstruction}
Suppose that $\Delta = 2d-3j=1$ where $(d,j)=(3m+2,2m+1)$ and
$m=2k$ is an even integer. Then the first obstruction to the
existence of a $\mathbb{D}_8$-equivariant map
(\ref{eqn:equi-map}), evaluated as an element of $\mathbb{Z}/4$,
is equal to
\begin{equation}\label{eqn:rezultat}
\theta = 2{{2k-1}\choose{k-1}}  \quad \mbox{\rm (mod 4)}.
\end{equation}
\end{prop}
\begin{proof}
By (\ref{eqn:cocycle}) a generator of the obstruction group is the
cocycle $\phi$ such that $\phi(x_0)=1, \phi(x_1)=-1$ where $x_0$
and $x_1$ are the (relative) fundamental classes of the
pseudomanifolds $X_0$ and $X_1$. Following the calculations (and
notation) from Section~\ref{sec:standard}, the homologically
relevant part of the geometric boundary of the pseudomanifold
$X_0$ has the following representation.
\begin{equation}\label{eqn:representation}
\partial X_0 = Z_0'\cup T_0' = (Z_0\cup \alpha Z_0\cup \gamma Z_0 \cup \alpha\gamma
Z_0)\cup (T_0\cup \beta T_0).
\end{equation}
The set $M:= Z_0'\cup T_0'$ is, up to a closed subset of high
codimension,  a closed oriented manifold. In light of
(\ref{eqn:cocycle-formula}) the obstruction $\theta$ can be
evaluated as the degree ${\rm deg}(f)$ where $f = \psi\vert_M$ is
the restriction to $M$ of an arbitrary $\mathbb{D}_8$-equivariant
map $\psi : F_{2d-2}\rightarrow S(U^{\oplus
j}\oplus\lambda^{\oplus j})$. Such a map clearly exists since
${\rm dim}(F_{2d-2})={\rm dim}(S(W^{\oplus j}))$. Interpreting the
degree as an algebraic count of points in the preimage of a
regular value, we observe that
$$
{\rm deg}(\psi) = {\rm deg}(\psi_{\vert Z_0'}) + {\rm
deg}(\psi_{\vert T_0'}).
$$
The computation of the degrees on the right is facilitated by the
equivariance of $\psi$ and existing $\mathbb{D}_8$-symmetries of
$X_0$ and $T_0$. In particular, the fact that $X_0$ is ``a half of
the manifold'' $S^{d-1}\times S^{d-1}\subset F_{2d-2}$, together
with Proposition~\ref{prop:exists}, implies that
\begin{equation}\label{eqn:rezultat-2}
{\rm deg}(\psi_{\vert Z_0'}) = 2{{2k-1}\choose{k-1}}  \quad
\mbox{\rm (mod 4)}.
\end{equation}
Similarly, ${\rm deg}(\psi_{\vert T_0'}) = 0$ follows from the
fact that $\beta$ acts on $T_0' = T_0\cup \beta T_0$ by changing
its orientation (equation (\ref{eqn:boundary})), while it keeps
the orientation of $S(W^{\oplus j})$ fixed (equation
(\ref{eqn:karakter})). This observation completes the proof of the
proposition.
\end{proof}

\section{Summary of proofs of main results}
\label{sec:summary}

\noindent
 {\bf Proof of Theorem~\ref{thm:main-1}.} By
 Proposition~\ref{prop:admissible}, Theorem~\ref{thm:main-1} is an
 immediate consequence of Theorem~\ref{thm:main-2}. \hfill $\square$

\medskip
\noindent
 {\bf Proof of Theorem~\ref{thm:main-2}.} By
 Proposition~\ref{prop:fundam-obstruction} the existence of an
 equivariant map (\ref{eqn:main-2}) implies the existence of the
 chain map (\ref{eqn:ladder}) where $C_\ast$ is the chain complex
 associated to an admissible filtration in the sense of
 Definition~\ref{defn:adm-filtration}.

In order to facilitate calculations we choose the standard
filtration on $X = S^n\times S^n$ described in
Section~\ref{sec:standard} and calculate the relevant fragment of
the associated chain complex $C_\ast = C_\ast(S^n\times S^n)$ in
Section~\ref{sec:fragment}.

By Proposition~\ref{prop:chain-obstruction} the first obstruction
$\theta$ to the existence of the chain map (\ref{eqn:ladder}) lies
in the cohomology group (\ref{eqn:alg-cohom-gp}) which is in our
case (see (\ref{eqn:spec-cohomology}) and
Remark~\ref{rem:oznaka-cohom-gp}) the group
$\mathcal{H}^{2n-1}_{\mathbb{D}_8}(S^n\times S^n; \mathcal{Z})$.
This group is evaluated in Proposition~\ref{prop:spectral}
(Section~\ref{sec:evaluation}) and found to be isomorphic to
$\mathbb{Z}/4$.

The obstruction class $\theta$ is evaluated in
Proposition~\ref{prop:obstruction} following the description of
the associated cocycle given in (\ref{eqn:cocycle-formula}). We
use the idea, described in greater detail in
Section~\ref{subsec:heuristic}, that in some situations we have
some freedom in choosing the map $f_n$ for evaluating the
obstruction cocycle $\theta(f_n)$ by the formula
(\ref{eqn:cocycle-formula}).

As an element of the group $\mathbb{Z}/4$ the obstruction class
$\theta$ is according to (\ref{eqn:rezultat}) equal to
 $$
 \theta = 2{{2k-1}\choose{k-1}}.
 $$
It is non-zero if and only if ${2k-1}\choose{k-1}$ is an odd
integer which is the case if and only if $k=2^l$. It follows that
in the case when $d$ is even the triple $(d,j,2)$ is admissible if
for some integer $l\geq 0$,
 $$
 d=3\cdot 2^{l+1}+2 \quad \mbox{\rm and} \quad j= 2\cdot
 2^{l+1}+1.
 %\tag*{$\square$}
 $$
This completes the proof of Theorem~\ref{thm:main-2}. \hfill
$\square$

\section{Appendix I: Mapping degree of equivariant maps }
\label{sec:appendix-I}

There is a general principle, see Kushkuley and Balanov
\cite{K-B}, equation (0.1) in Section~0.3., relating the mapping
degrees of two $G$-equivariant maps,
\[
f, g : M^n \rightarrow S(V)
\]
where $G$ is a finite group, $M^n$ is a not necessarily free,
closed, oriented $G$-manifold, and $S(V)$ is the unit sphere in a
real, $(n+1)$-dimensional $G$-vector space $V$. The principle says
that there exists a relation
\begin{equation}\label{eqn:rel}
 \pm({\rm deg}(f) - {\rm deg}(g)) = \sum_{j=1}^k a_j\vert
G/H_j\vert
\end{equation}
for some integers $a_j$, where $\mathcal{H}=\{H_1,\ldots, H_k\}$
is a list of isotropy groups corresponding to orbit types of
$M^n$, provided the {\em orientation characters} of $M$ and $V$
are the same in the sense that each $g\in G$ either changes
orientations of both $M$ and $V$, or keeps them both unchanged.

\medskip
In some (favorable) situations the ``local degrees'' $a_j$ vanish
if $H_j\neq \{e\}$ is a non-trivial subgroup of $G$, in which case
the equation (\ref{eqn:rel}) implies the congruence
\begin{equation}\label{eqn:cong}
{\rm deg}(f) \equiv {\rm deg}(g) \quad \mbox{ ({\rm mod} } \vert G
\vert).
\end{equation}
Here we record one of these favorable situations applying to the
case of the dihedral group $\mathbb{D}_8$ of order $8$ acting on
the manifold $M = S^m\times S^m$.

\begin{thm}\label{thm:ovde-glavna}
Suppose that $V$ is a real $\mathbb{D}_8$-vector space of
dimension $(2m+1)$ such that the isotropy space $V_\gamma$,
corresponding to the element $\gamma\in \mathbb{D}_8$, has
dimension $\geq m+2$. Assume that the representation $V$ has the
same orientation character as the space $S^m\times S^m$. Then for
each pair $f,g : S^m\times S^m \rightarrow S(V)$ of\,
$\mathbb{D}_8$-equivariant maps, the associated mapping degrees
satisfy the following congruence relation,
\begin{equation}\label{eqn:cong-8}
{\rm deg}(f) \equiv {\rm deg}(g) \quad \mbox{ {\rm (mod} } 8)
\end{equation}
\end{thm}

\begin{proof} The action of $\mathbb{D}_8$ on $S^m\times S^m$ is free away
from the two spheres $S_\gamma = \{(x,y)\in S^m\times S^m \mid
x=y\}$ and $S_{\alpha\beta\gamma} = \{(x,y)\in S^m\times S^m \mid
x=-y\}$.

Let $f_1$ and $g_1$ be the restrictions of $f$ (respectively $g$)
on their union $T = S_\gamma \cup S_{\alpha\beta\gamma}$. The
space $T$ is $\mathbb{D}_8$-invariant and our initial objective is
to show that there exists a $\mathbb{D}_8$-equivariant homotopy
$F_1 : T\times [0,1] \rightarrow V$ between $f_1$ and $g_1$.

We apply the Proposition~I.7.4 from \cite{tDieck87} (page 52)
which says that it is sufficient to construct a
$\mathbb{Z}/2$-equivariant homotopy $F_1' : S_\gamma\times [0,1]
\rightarrow V_\gamma$ between the restrictions $f_1' = f_1\vert
S_\gamma$ and $g_1' = g_1\vert S_\gamma$, where $\mathbb{Z}/2 =
N(\langle\gamma\rangle)/\langle\gamma\rangle$ is the Weyl group of
$\langle\gamma\rangle$. Since this $\mathbb{Z}/2$-action on
$S_\gamma$ is free,  the existence of the homotopy $F_1'$ follows
immediately from the assumption,
\[
{\rm dim}(S_\gamma) = m \leq {\rm dim}(S(V_\gamma)) -1 .
\]
The homotopy  $F_1$ can be extended to a
$\mathbb{D}_8$-equivariant homotopy,
\[
F : S^m\times S^m \rightarrow D(V)
\]
between $f$ and $g$ where $D(V)={\rm Cone}(S(V))$ is the unit ball
in $V$. Moreover, since the action of $\mathbb{D}_8$ on
$(S^m\times S^m)\setminus (S_\gamma\cup S_{\alpha\beta\gamma})$ is
free, we can assume that $F$ is smooth away from $S_\gamma\cup
S_{\alpha\beta\gamma}$ and transverse to $0\in V$.

\medskip
The set $Z(F):=F^{-1}(0)$ is finite, $G$-invariant and a union of
free orbits. For each $x\in Z(F)$ choose an open ball $O_x\ni x$
such that $O:= \cup_{x\in Z(F)}~O_x$ is $G$-invariant and $O_x\cap
O_y\neq\emptyset$ for $x\neq y$. Let $S_x :=\partial(O_x)\cong
S^{2m}$ be the boundary of $O_x$.

\medskip
Let $N := (M\times [0,1])\setminus O$, $M_0:= M\times \{0\}$ and
$M_1:= M\times \{1\}$. By construction there is a relation among
(properly oriented) fundamental classes,
\begin{equation}\label{eqn:homoloska-relacija}
  [M_1] - [M_0] = \sum_{x\in Z(F)} [S_x]
\end{equation}
inside the homology group $H_{2m+1}(N,\mathbb{Z})$. The map
$$F_\ast : H_{2m+1}(N,\mathbb{Z})\rightarrow
H_{2m+1}(V\setminus\{0\},\mathbb{Z})\cong \mathbb{Z}$$ maps the
left hand side of the relation (\ref{eqn:homoloska-relacija}) to
the difference of degrees, ${\rm deg}(f) - {\rm deg}(g)$. The
right hand side is mapped to an element divisible by $8$ since by
assumption the orientation characters of manifolds $M\times [0,1]$
and $V$ are the same.
\end{proof}

\subsection{Corollaries}\label{sec:corollaries}

\begin{cor}\label{cor:prvi} Suppose that
$f, g : S^m\times S^m \rightarrow S(\lambda^{\oplus (2m+1)})$ are
$\mathbb{D}_8$-equivariant maps where $m=2k$ is an even integer.
Then
$$
\mbox{\rm deg}(f) \equiv \mbox{\rm deg}(g) \quad (\mbox{\rm mod}\,
8).
$$
\end{cor}

\begin{proof} The orientation character of both $S^m\times S^m$ and
$S(\lambda^{\oplus (2m+1)})$ is given by the formula $(\alpha,
\beta, \gamma) = (-1,-1,+1)$, see (\ref{eqn:one-dim-rep}) in
Section~\ref{sec:notation} and (\ref{eqn:or-character}) in
Section~\ref{sec:fragment}. Since the whole space $\lambda^{\oplus
(2m+1)}$ is fixed by $\gamma$ the result is an immediate
consequence of Theorem~\ref{thm:ovde-glavna}.
\end{proof}

\begin{cor}\label{cor:drugi} Suppose that
\begin{equation}\label{eqn:exists-bis}
f, g : S^{d-1}\times S^{d-1}\rightarrow S(W^{\oplus j})
\end{equation}
are $\mathbb{D}_8$-equivariant maps where $(d,j)=(6k+2,4k+1)$ for
some integer $k> 0$. Then,
\[
\mbox{\rm deg}(f) \equiv \mbox{\rm deg}(g) \quad (\mbox{\rm mod}\,
8).
\]
\end{cor}

\begin{proof} In light of (\ref{eqn:or-character}) and (\ref{eqn:karakter})
the orientation character of both $S^{d-1}\times S^{d-1}$ and
$W^{\oplus j}$ is given by
\[
(\alpha, \beta, \gamma) = (+1, +1, -1).
\]
The dimension $D$ of the isotropy space $W^{\oplus j}_\gamma$ is
$2j = 8k+2$ so the dimension requirement $D\geq d+1$ from
Theorem~\ref{thm:ovde-glavna} reduces to $8k+2\geq 6k+3$ which is
satisfied if $k>0$.
\end{proof}

\section{Appendix II: Obstructions, filtrations and chain complexes}
\label{sec:appendix}

\subsection{The obstruction exact sequence}\label{subsec:obstruction}
For a review of equivariant obstruction theory, which includes an
exposition of $G$-$CW$-complexes and the $G$-cellular
approximation theorem, the reader is referred to
\cite[Chapter~II]{tDieck87}.

One of the central results in the area is the following {\em
obstruction exact sequence}.

\begin{thm}\label{thm:obstruction-fund}
Suppose that $X$ is a free $G$-$CW$-complex and that $Y$ is
$n$-simple $G$-space for a fixed integer $n\geq 1$. Then there
exists an exact sequence,
\begin{equation}\label{eqn:obstruction_fund}
[X^{(n+1)},Y]_G \longrightarrow \mbox{\rm Image}\{ [X^{(n)},Y]_G
\rightarrow [X^{(n-1)},Y]_G \} \longrightarrow
\mathcal{H}^{n+1}_G(X,\pi_n(Y))
\end{equation}
\end{thm}

In many applications $Y = S(V)$ is a $G$-sphere $S(V)\cong S^n$
for some real $G$-module $V$. In that case
Theorem~\ref{thm:obstruction-fund} has the following important
corollary.

\begin{cor}\label{cor-fund}
Suppose that $X$ is a free $G$-$CW$-complex. Let $Y=S(V)=S^n$ be
an $n$-dimensional $G$-sphere associated to a real
$G$-representation $V$  $(n\geq 2)$. Then
(\ref{eqn:obstruction_fund}) reduces to,
\begin{equation}\label{eqn:reduces-to}
[X^{(n+1)},S(V)]_G \longrightarrow \{\ast\} \longrightarrow
\mathcal{H}^{n+1}_G(X,\mathcal{Z})
\end{equation}
where $\mathcal{Z}$ is the orientation character of $S(V)$ and
$\{\ast\}$ is a singleton.
\end{cor}
The exactness of the sequence (\ref{eqn:reduces-to}) means that
there exists a single element of the obstruction group
$\mathcal{H}^{n+1}_G(X,\mathcal{Z})$ (the image of $\ast$) which
is zero if and only if there exists a $G$-equivariant map $f :
X^{(n+1)}\rightarrow S(V)$.

\subsection{Admissible filtrations}\label{subsec:admissible}

Our objective is to extend the applicability of basic obstruction
theory, as outlined in Section~\ref{subsec:obstruction}, by
introducing more general filtrations which do not necessarily
arise from a $G$-$CW$-structure on $X$. The reader may keep in
mind the $G$-manifold complexes introduced in
Section~\ref{sec:methods} as a guiding example of such
filtrations.

\begin{defn}\label{defn:adm-filtration} Let $X$ be a not necessarily free $G$-space
which admits a $G$-invariant triangulation ($CW$-structure)
turning $X$ into a simplicial complex ($CW$-complex) of dimension
$d\geq n+1$. Let $X^{(k)}$ be the associated $k$-skeleton. A
finite filtration
\begin{equation}\label{eqn:adm-filtration}
\emptyset = X_{-1}\subset X_0\subset X_1\subset \ldots \subset
X_{n-1}\subset X_n \subset X_{n+1} \subset \ldots \subset X_d = X
\end{equation}
is called {\em admissible} if the following condition is
satisfied,
\begin{enumerate}
 \item $X_k$ is a $G$-invariant subcomplex of $X^{(k)}$ for each $k$.
\end{enumerate}
Let $C_\ast(X)$ be the $G$-chain complex associated to the
filtration (\ref{eqn:adm-filtration}) where $C_k(X):=
H_k(X_{k},X_{k-1}; \mathbb{Z})$.  The associated cohomology groups
with coefficients in a $\mathbb{Z}[G]$-module $M$ are sometimes
referred to as {\em special equivariant cohomology groups} and
denoted by
\begin{equation}\label{eqn:spec-cohomology}
\mathcal{H}^\ast_G(X; M).
\end{equation}
An admissible filtration (\ref{eqn:adm-filtration}) is said to be
{\em free in dimension} $k$ if,
\begin{enumerate}
 \item[(2)] $C_k = C_k(X):= H_k(X_{k},X_{k-1}; \mathbb{Z})$ is a free
 $\Lambda$-module, where $\Lambda := \mathbb{Z}[G]$
 is the group ring of the group $G$.
\end{enumerate}
\end{defn}

\begin{rem}\label{rem:caveat}
The reader should keep in mind that the cohomology groups
(\ref{eqn:spec-cohomology}) are functors of a filtered space $X$,
not the space alone. This is in agreement with the approach in
\cite[page 112]{tDieck87} where the corresponding equivariant
cohomology groups depend on a given $G$-$CW$-structure.
\end{rem}

\begin{rem}\label{rem:free}
The structure of a simplicial complex on a $G$-space $X$ plays an
auxiliary role and in intended applications one starts directly
with a filtration (\ref{eqn:adm-filtration}), tacitly assuming
that it can be ``triangulated''. The most natural is the situation
when $X$ is a $G$-manifold complex where the constituent manifolds
are semialgebraic sets (as in Section~\ref{sec:standard}) and this
condition is automatically satisfied. The condition (2) is
evidently not necessarily satisfied if $X$ is a free $G$-space
with an admissible filtration.
\end{rem}

We assume that the target $G$-space $Y$ is also filtered by a
filtration
\begin{equation}\label{eqn:cw-filtration}
\emptyset = Y_{-1}\subset Y_0\subset Y_1\subset \ldots \subset Y_n
\subset Y_{n+1} \subset \ldots \subset Y_\nu = Y
\end{equation}
arising from some, not necessarily free, $G$-$CW$-structure on
$Y$. Let $D_\ast = (\{D_k\}_{0}^{\nu},\partial)$ be the associated
cellular chain complex, $D_k := H_k(Y_k, Y_{k-1})$.

\medskip
The following proposition allows us to reduce the problem of the
existence of $G$-equivariant maps between $X$ and $Y$, to the
question of the existence of chain maps between the associated
chain complexes of $\Lambda = \mathbb{Z}[G]$ modules.

\begin{prop}\label{prop:fundam-obstruction}
Suppose that $X$ is a $d$-dimensional $G$-space with an {\em
admissible filtration} (\ref{eqn:adm-filtration})
(Definition~\ref{defn:adm-filtration}). Suppose that $Y$ is a
$G$-$CW$-complex and let (\ref{eqn:cw-filtration}) be its
associated filtration by skeletons. Then if there exists a
$G$-equivariant map $f : X\rightarrow Y$, there exists also a
chain map $f_\ast : C_\ast (X) \rightarrow D_\ast(Y)$ of the
associated, augmented chain complexes,
\begin{equation}\label{eqn:ladder}
\xymatrix@=20pt{
 ..\ar[r] & C_{n+1} \ar[r]^{\partial} \ar[d]^{f_{n+1}} &
C_n \ar[r]^-{\partial}\ar[d]^{f_n} & C_{n-1}
\ar[r]\ar[d]^{f_{n-1}}
 & \ldots \ar[r] & C_1 \ar[r]^{\partial}\ar[d]^{f_1} & C_0 \ar[r]\ar[d]^{f_0} & \mathbb{Z}\ar[d]^{\cong}\ar[r] & 0\\
 ..\ar[r] & D_{n+1} \ar@{>}[r]_{\partial} & D_n \ar[r]_{\partial} & D_{n-1} \ar[r]
 & \ldots \ar[r] & D_1 \ar[r]_{\partial} & D_0 \ar[r]_{\partial} & \mathbb{Z}\ar[r] & 0
 }
\end{equation}
\end{prop}

\begin{proof} If there exists a $G$-equivariant map $f : X\rightarrow
Y$ then by the {\em cellular approximation theorem} (Theorem~2.1
in \cite[Section~II.2]{tDieck87}) there exists a cellular map $g :
X\rightarrow Y$ which is $G$-homotopic to $f$. In follows that
$g(X_{k})\subset g(X^{(k)})\subset Y^{(k)}=Y_k$ which in turn
implies the existence of the chain map (\ref{eqn:ladder}).
\end{proof}

\subsection{Algebraic description of the obstruction}
\label{subsec:alg-description}

Proposition~\ref{prop:fundam-obstruction} allows us to reduce the
topological problem of the existence of equivariant maps to an
algebraic problem of finding a chain map. This in turn leads to
algebraic counterparts of Theorem~\ref{thm:obstruction-fund} and
Corollary~\ref{cor-fund}.

\begin{prop}\label{prop:chain-obstruction}
Suppose that $C_\ast :=\{C_k\}_{k=-1}^{d}$ and $D_\ast
:=\{D_k\}_{k=-1}^{d}$ are finite chain complexes of $\Lambda =
\mathbb{Z}[G]$ modules where $C_{-1}\cong D_{-1}\cong \mathbb{Z}$.
Suppose that the chain map $F_{n-1}:=(f_j)_{j=-1}^{n-1} :
\{C_k\}_{k=-1}^{n-1}\rightarrow \{D_k\}_{k=-1}^{n-1}$ exists and
is fixed in advance ($n+1\leq d$). Suppose that $F_{n-1}$ can be
extended one step further, i.e.\ that there exists a homomorphism
$ f_n : C_n\rightarrow D_n$ such that $\partial\circ  f_n =
f_{n-1}\circ\partial$. Then the obstruction to the existence of a
chain map (\ref{eqn:ladder}), $F_{n+1}:=(f_j)_{j=-1}^{n+1} :
\{C_k\}_{k=-1}^{n+1}\rightarrow \{D_k\}_{k=-1}^{n+1}$, extending
the chain map $F_{n-1}$ (with the modification of $f_n$ if
necessary) is a well defined element $\theta$ of the cohomology
group
\begin{equation}\label{eqn:alg-cohom-gp}
H^{n+1}(C_\ast; H_n(D_\ast)) = H_{n+1}(\mbox{\rm Hom}(C_\ast,
H_n(D_\ast))).
\end{equation}
Moreover, $\theta$ is represented by the cocycle
\begin{equation}\label{eqn:cocycle-formula}\theta( f_n) :
C_{n+1}\stackrel{\partial}{\longrightarrow} C_n \stackrel{
f_n}{\longrightarrow} Z_n(D_\ast) \stackrel{\pi}{ \longrightarrow}
H_n(D_\ast).\end{equation} The vanishing of $\theta$ is not only
necessary but also sufficient for the existence of the chain map
$F_{n+1}$ (\ref{eqn:ladder}) if  $C_n$ and $C_{n+1}$ are free (or
projective)  modules.
\end{prop}

\begin{proof} The homomorphism $f_n\partial  : C_{n+1}\rightarrow D_{n}$
has the following properties:

\begin{enumerate}
 \item[(1)] ${\rm Image}(f_n\partial) \subset
 Z_n({D_\ast}) := {\rm Ker}(D_n\stackrel{\partial}\rightarrow
 D_{n-1})$.
 \item[(2)] ${\rm Image}(f_n\partial ) \subset
 B_n({D_\ast}):= {\rm Image}(D_{n+1}\stackrel{\partial}\rightarrow
 D_{n})$  if and only if there exists $f_{n+1}$ such that
 $\partial\circ f_{n+1}= f_n\circ\partial$
%(here we use the projectivity of the module $C_{n+1}$).
(provided $C_{n+1}$ is a projective module).
\end{enumerate}
In other words if $\theta(f_n)=\pi\circ f_n\circ\partial$ is the
homomorphism
$$\theta(f_n) :
C_{n+1}\stackrel{\partial}{\longrightarrow} C_n
\stackrel{f_n}{\longrightarrow} Z_n(D_\ast) \stackrel{\pi}{
\longrightarrow} H_n(D_\ast)$$ then $f_{n+1}$ exists if and only
if $\theta(f_n) = 0$.

\begin{enumerate}
\item[(3)] $\theta(f_n) \in Hom(C_\ast, H_n(D_\ast))$ is a
cocycle.
\end{enumerate}
If $\theta := [\theta(f_n)]$ is the associated cohomology class
let us show that it has the properties claimed in the proposition.

Suppose that the chain map $F_{n+1}$ exists where $f_n$ is
replaced by a homomorphism $f_n' : C_n\rightarrow D_n$ such that
$f_{n+1}\circ\partial = \partial\circ f_n'$ and
$f_{n}'\circ\partial = \partial\circ f_{n-1}$ (in particular
$\theta(f_n')=0$). Let $h_1 = f_n-f_n' : C_n\rightarrow D_n$.
Since $\partial\circ f_n = \partial\circ f_n'$ we observe that
${\rm Image}(h_1)\subset Z_n(D_\ast)$. Let $h = \pi\circ h_1 :
C_n\rightarrow H_n(D_\ast)$. It follows that
$$
\theta(f_n)=\theta(f_n)-\theta(f_n') = \theta(h_1) = \pi\circ
h_1\circ\partial = h\circ\partial = \delta(h)
$$
which implies that $[\theta(f_n)]=0\in H^{n+1}(C_\ast ;
H_n(D_\ast))$.

\medskip
Conversely if $\theta = 0$ then $\theta(f_n)$ is a coboundary,
i.e.\ $\theta(f_n) = \delta(h) = h\circ\partial$ for some
homomorphism $h : C_n \rightarrow H_n(D_\ast)$. Since the module
$C_n$ is projective, there exists a homomorphism $h_1 : C_n
\rightarrow Z_n(D_\ast)$ such that $h = \pi \circ h_1$, hence
\begin{equation}\label{eqn:imamo-relaciju}
\theta(f_n) = \pi\circ f_n \circ \partial = \pi\circ h_1 \circ
\partial : C_{n+1} \rightarrow H_n(D_\ast).
\end{equation}
Let $f_n' := f_n - h_1$. Since ${\rm Image}(h_1)\subset
Z_n(D_\ast)$ we observe that $\partial\circ f_n' = \partial\circ
f_n$ which implies that $\partial\circ f_n' =
f_{n-1}\circ\partial$.

From the equality $\theta(f_n')=\pi\circ f_n'\circ
\partial =0$ we deduce that ${\rm
Image}(f_n'\circ\partial)\subset B_n(D_\ast)$. Since $C_{n+1}$ is
a projective module, there exists a homomorphism $f_{n+1}:
C_{n+1}\rightarrow D_{n+1}$ such that $\partial\circ f_{n+1} =
f_n'\circ\partial$. This completes the construction of the chain
map $F_{n+1}$.
\end{proof}

\begin{rem}\label{rem:oznaka-cohom-gp}
If the chain complex $C_\ast = C_\ast(X)$ arises from a fixed
admissible filtration (\ref{eqn:adm-filtration})
(Definition~\ref{defn:adm-filtration}) then the group
$H^{n+1}(C_\ast; M)$ (where $M = H_n(D_\ast)$) is nothing but the
(special) equivariant cohomology group $\mathcal{H}^{n+1}_G(X; M)$
which appears in line (\ref{eqn:spec-cohomology}).
\end{rem}

\begin{exmp}\label{ex:4-torzija} The first obstruction $\theta$ to the
existence of a chain map between the following chain complexes
lies in the group $H^{2}(C_\ast; H_1(D_\ast))\cong \mathbb{Z}/4$.
Assuming
\begin{equation}\label{eqn:4-torzija}
\xymatrix@=20pt{
 \mathbb{Z}[\mathbb{Z}/2] \ar[r]^{1-\omega} \ar[d]^{f_3} & \mathbb{Z}[\mathbb{Z}/2]
 \ar[r]^{2(1+\omega)} \ar[d]^{f_2} & \mathbb{Z}[\mathbb{Z}/2]
 \ar[r]^{1-\omega} \ar[d]^{f_1} & \mathbb{Z}[\mathbb{Z}/2]
 \ar[r] \ar[d]^{f_0} & \mathbb{Z} \ar[r] \ar[d]^{\cong} &  0 \\
0 \ar[r] &   0 \ar[r] & \mathbb{Z}[\mathbb{Z}/2] \ar[r]_{1-\omega}
& \mathbb{Z}[\mathbb{Z}/2] \ar[r] & \mathbb{Z} \ar[r] & 0
 }
\end{equation}
that both $f_0$ and $f_1$ are the identity maps a simple
calculation shows that $\theta = 2$.
\end{exmp}

\begin{rem}\label{rem:Lense_space}
The cyclic group $\mathbb{Z}/4=\{1,\omega,\omega^2,\omega^3 \}$
acts on the sphere $S^3 = S^1\ast S^1\subset \mathbb{C}^2$ by
rotating each of the circles $S^1$ through the angle of $90^\circ$
with the Lens space $L(4,1)=S^3/(\mathbb{Z}/4)$ as the quotient.
The partial quotient $S^3/H$ where $H=\{1, \omega^2\}$ is the
projective space $\mathbb{R}P^3$ which inherits a $\mathbb{Z}/2$
action form the group $(\mathbb{Z}/4)/H$. It is not difficult to
see that the first line of (\ref{eqn:4-torzija}) is an associated
chain complex obtained as a quotient from the standard
$\mathbb{Z}/4$-invariant $CW$-structure on $S^3$. In light of
Proposition~\ref{prop:fundam-obstruction} the non-triviality of
the obstruction $\theta$, as calculated in
Example~\ref{ex:4-torzija}, guarantees that there does not exist a
$\mathbb{Z}/2$-equivariant map $f :
(\mathbb{R}P^3)^{(2)}\rightarrow S^1$, where $S^1$ has the
antipodal action and $(\mathbb{R}P^3)^{(2)}$ is the $2$-skeleton
of $\mathbb{R}P^3$.
\end{rem}

\begin{exmp}\label{ex:non-example}
Suppose we want to show that there does not exist a
$\mathbb{Z}/2$-equivariant map $f : S^{n+1}\rightarrow S^{n}$
(Borsuk-Ulam theorem) by the methods of this paper. We initially
choose an admissible filtration $\{F_k\}_{k=0}^{n+1}$ of $S^{n+1}$
(in the sense of Definition~\ref{defn:adm-filtration}) by defining
$F_0=\ldots=F_{n-1}=S^0, F_n=S^n$ and $F_{n+1}=S^{n+1}$, where
$S^0\subset S^n$ are $\mathbb{Z}/2$-invariant subspheres of
$S^{n+1}$. Then $C_n \cong H_n(S^n, S^0; \mathbb{Z})\cong
\mathbb{Z}$ and easy calculation shows that the associated
obstruction $\theta$ is $0$, meaning that this filtration is not
well adopted for this problem. If we modify the filtration by
choosing $F_{n-1}$ to be a $\mathbb{Z}/2$-invariant sphere
$F_{n-1}=S^{n-1}$ (where $S^0\subset S^{n-1}\subset S^n$) then a
direct calculations shows that $\theta\neq 0$.
\end{exmp}

\begin{rem}
The last part of Proposition~\ref{prop:chain-obstruction},
claiming that $\theta$ is a complete obstruction provided $C_n$
and $C_{n+1}$ are free modules, is a motivation for isolating
admissible filtrations free in selected dimensions, cf.\ condition
(2) in Definition~\ref{defn:adm-filtration}. Note that the
freeness of $C_n$ and $C_{n+1}$ is a condition that can be
satisfied even if $X$ is not a free $G$-space, e.g.\ if the
corresponding set of fixed points $X^H$ has a codimension $\geq 2$
for each subgroup $H\neq \{e\}$.
\end{rem}

\subsection{Heuristics for evaluating the obstruction $\theta$}
\label{subsec:heuristic}

In many cases the chain map $F_{n-1}=(f_j)_{j=-1}^{n-1}$, which in
Proposition~\ref{prop:chain-obstruction} serves as an input for
calculating the obstruction $\theta$, is unique up to a chain
homotopy. This happens for example when $D_\ast$ is a chain
complex associated to a $G$-sphere $Y$ of dimension $n$.

In this case one can inductively build a `ladder of maps'
(\ref{eqn:ladder}) in order to find a representative of the chain
homotopy class of the chain map $[F_{n-1}]$ needed for the
evaluation of $\theta$. A very instructive explicit example of
this calculation can be found in \cite{Kriz}.

In practise one can bypass these calculations by constructing and
using instead an arbitrary $G$-equivariant map $\phi_n : X_n
\rightarrow Y$. In the case of $G$-manifold complexes
(Section~\ref{sec:methods}) the group $C_n = H_n(X_n, X_{n-1};
\mathbb{Z})$ is generated by the (relative) fundamental classes of
manifolds with boundary and the calculation of the map $\pi \circ
f_n : C_n \rightarrow H_n(D_\ast) \cong H_n(Y) \cong H_n(S^n)$ in
(\ref{eqn:cocycle-formula}) is reduced to the calculation of the
corresponding mapping degrees.

Examples of calculations which essentially follow this procedure
can be found in \cite[Section~4]{Guide2}.

\subsection{Computational topology and effective obstruction theory}
\label{sec:app-comput-top}

The problem of  calculating topological obstructions to the
existence of equivariant maps is identified in \cite{BVZ} as one
of the questions of great relevance for computational topology.
The focus is naturally on those features of the obstruction
problem where topology and computational mathematics interact in
an essential way. This brings forward algorithmic aspects of the
question emphasizing explicit procedures suitable for
semiautomatic and/or large scale calculations. Here we
recapitulate and briefly summarize some of the general ideas used
in the proof of Theorem~\ref{thm:main-2} which may be of some
independent interest in the development of the effective
obstruction theory.

\begin{enumerate}
\item[(1)] One works with manifold $G$-complexes
(Section~\ref{sec:methods}) which are more general and often more
economical than $G$-$CW$-complexes.

\item[(2)] Given a $G$-space (manifold) $X$, the associated
$G$-manifold complex arises through the iteration of the
`fundamental domain - geometric boundary' procedure (see the
diagram in Section~\ref{sec:standard}).

\item[(3)] The fact that the generators of $G$-modules are
fundamental classes (Section~\ref{subsec:heuristic}) allows us to
evaluate boundaries and chain maps as the mapping degrees
(Section~\ref{sec:fragment}, \ref{sec:evaluation} and
\ref{sec:obstr-evaluation}).

\item[(4)] The emphasis in the basic set-up of the obstruction
theory (Appendix, Section~\ref{sec:appendix}) is on `admissible
filtrations' (Proposition~\ref{prop:fundam-obstruction}) and chain
complexes, rather than spaces,
(Proposition~\ref{prop:chain-obstruction}).
\end{enumerate}

\subsection*{Acknowledgments}  It is a pleasure to thank P.~Landweber,
S.~Melikhov, and members of the Belgrade
CGTA\footnote{Combinatorics in Geometry, Topology, and Algebra.}
seminar for useful comments on a preliminary version of this paper
(April 10, 2011).

\end{document}